\documentclass[reqno,11pt]{amsart}

\usepackage[T1]{fontenc}
\usepackage{lmodern}
\usepackage{microtype}
\usepackage[letterpaper,hmargin=1.25in,vmargin=0.9in]{geometry}

\usepackage{mathtools}
\usepackage{amssymb}
\usepackage[unicode,hidelinks]{hyperref}
\allowdisplaybreaks[1]
\newcommand{\R}{\mathbb{R}}
\newcommand{\Z}{\mathbb{Z}}
\DeclareMathOperator{\conv}{conv}
\DeclareMathOperator{\lin}{lin}
\DeclareMathOperator{\rank}{rank}

\newcommand{\Dcount}[3]{D_{#1}(#2,#3)}
\newcommand{\Ncount}[2]{N_{#1}(#2)}
\newcommand{\doi}[1]{\url{https://doi.org/#1}}

\theoremstyle{plain}
\newtheorem{theorem}{Theorem}[section]
\newtheorem{proposition}[theorem]{Proposition}
\newtheorem{lemma}[theorem]{Lemma}
\newtheorem{corollary}[theorem]{Corollary}
\theoremstyle{remark}
\newtheorem{remark}[theorem]{Remark}

\numberwithin{equation}{section}

\title[Rank-Two Discrete Monotonicity]{The Discrete Monotonicity Property in Rank Two}

\author{Fei Xue}
\address{School of Mathematical Sciences, Nanjing Normal University, Nanjing 210046, China}
\email{05429@njnu.edu.cn}

\subjclass[2020]{Primary 11H06, 52C05; Secondary 52C07, 05A20}
\keywords{discrete monotonicity, lattice-convex sets, residue classes, finite tori, successive minima}
\date{October 28, 2026}

\hypersetup{
  pdftitle={The Discrete Monotonicity Property in Rank Two},
  pdfauthor={Fei Xue},
  pdfsubject={Discrete monotonicity for subgroups of rank at most two},
  pdfkeywords={discrete monotonicity, lattice-convex sets, residue classes, finite tori, successive minima}
}

\begin{document}

\begin{abstract}
The Betke--Henk--Wills conjecture predicts a bound for lattice-point counts in terms of successive minima. Malikiosis introduced the discrete monotonicity property for normalized residue counts as a sufficient condition for this bound and its extension to nonsymmetric convex bodies. We prove that every compact convex set satisfies this property for lattice subgroups of rank at most two, in arbitrary ambient dimension.
\end{abstract}

\maketitle

\section{Introduction}\label{sec:introduction}

For an $o$-symmetric convex body $C=-C\subseteq\R^d$ and a full-rank lattice $\Lambda$, the $j$th successive minimum $\lambda_j(C,\Lambda)$ is the least $\lambda>0$ for which $\lambda C$ contains $j$ linearly independent vectors of $\Lambda$. For an arbitrary convex body $K\subseteq\R^d$, Malikiosis's extension of the Betke--Henk--Wills conjecture asserts that
\begin{equation}\label{eq:bhw}
    \#(K\cap\Lambda) \leq\prod_{j=1}^{d} \left(\left\lfloor \frac{2}{\lambda_j\bigl((K-K)/2,\Lambda\bigr)} \right\rfloor+1\right).
\end{equation}
For $K=-K$, this is the original conjecture of Betke, Henk, and Wills \cite[Conjecture~2.1]{BetkeHenkWills1993}, who proved it in dimension two \cite[Theorem~2.2]{BetkeHenkWills1993}. Malikiosis proved \eqref{eq:bhw} in dimensions at most three \cite[Corollary~4.1.1]{Malikiosis2012}. Tointon \cite{Tointon2024} obtained improved bounds in general dimension.

Malikiosis introduced a discrete monotonicity property that compares normalized residue counts at different moduli. His criterion for \eqref{eq:bhw} uses only subgroups of rank below the ambient dimension \cite[proof of Theorem~4.2.1]{Malikiosis2012}, so our rank-two result recovers the known three-dimensional bound. For a compact convex set $K\subseteq\R^d$, a lattice $\Lambda\subseteq\R^d$, and a subgroup $\Gamma\leq\Lambda$, write
\begin{equation}\label{eq:D-definition}
    \Dcount{\Lambda}{K}{r\Gamma} :=\#\{x+r\Gamma:x\in K\cap\Lambda\} \qquad(r\in\Z_{>0}).
\end{equation}
Thus $\Dcount{\Lambda}{K}{r\Gamma}$ counts the cosets of $r\Gamma$ containing a lattice point of $K$. The set $K$ has the \emph{discrete monotonicity property} with respect to $\Lambda$ if
\begin{equation}\label{eq:DMP}
    r\longmapsto \frac{\Dcount{\Lambda}{K}{r\Gamma}}{r^{\rank\Gamma}}
\end{equation}
is nonincreasing on the positive integers for every subgroup $\Gamma\leq\Lambda$ \cite[Definition~4.2.1]{Malikiosis2012}. This is a discrete counterpart of the quotient-volume monotonicity in classical proofs of Minkowski's second theorem \cite{BambahWoodsZassenhaus1965}.

\begin{theorem}[Rank-two discrete monotonicity]\label{thm:main}
Let $K\subseteq\R^d$ be compact and convex, let $\Lambda\subseteq\R^d$ be a lattice, and let $\Gamma\leq\Lambda$ have rank $i\leq2$. Then
\begin{equation}\label{eq:main}
    \frac{\Dcount{\Lambda}{K}{r\Gamma}}{r^i} \geq \frac{\Dcount{\Lambda}{K}{s\Gamma}}{s^i}
\end{equation}
for all positive integers $r\leq s$.
\end{theorem}

By the reductions and initial cases in Section~\ref{sec:reductions}, it suffices to compare moduli $r$ and $r+1$ for planar lattice-convex sets. We first choose a lattice-convex subset containing exactly one point from each represented class at the larger modulus (Lemma~\ref{lem:terminal-peeling}). A suitable lattice basis then places the points of each class at the smaller modulus among four corners of a lattice parallelogram (Corollary~\ref{cor:four-block}). Unless all classes at the smaller modulus are represented, we select lattice segments whose number equals the difference between the residue counts at moduli $r+1$ and $r$. Their residue images are distinct lines in at most three directions (Section~\ref{sec:lattice-chords}). By Proposition~\ref{prop:completion}, these lines are contained in a cover of the finite torus with $2r+1=(r+1)^2-r^2$ lines. Lemma~\ref{lem:critical-cover-expansion} bounds the number of points covered by the selected lines, proving the comparison from $r$ to $r+1$.
\section{Preliminaries}\label{sec:reductions}

Throughout, a lattice has full rank in its real span. For a finite subset $A$ of a lattice $L$, put
\begin{equation}\label{eq:N-definition}
    \Ncount{r}{A;L}:=\#\{a+rL:a\in A\} \qquad(r\in\Z_{>0}).
\end{equation}
This is the number of distinct cosets of $rL$ meeting $A$. We omit $L$ when the ambient lattice is fixed. We call $A$ \emph{lattice-convex} if $A=\conv(A)\cap L$. If $A=K\cap L$ is nonempty and $P=\conv(A)$, then $P\subseteq K$ and $P\cap L=A$. Thus residue counts for compact convex sets reduce to those for finite lattice-convex sets. If $A=\varnothing$, all residue counts are zero.

A nonzero vector of $L$ is \emph{primitive} if it is not an integer multiple $kv$ with $v\in L$ and $k\geq2$.

\begin{lemma}[Affine-coset reduction]\label{lem:affine-cosets}
Let $K\subseteq\R^d$ be compact and convex, let $\Lambda\subseteq\R^d$ be a lattice, and let $\Gamma\leq\Lambda$. For each coset $\xi\in\Lambda/\Gamma$, choose $z_\xi\in\xi$ and put
$$A_\xi:=(K\cap\xi)-z_\xi\subseteq\Gamma.$$
Every nonempty $A_\xi$ is lattice-convex in $\Gamma$, only finitely many are nonempty, and
\begin{equation}\label{eq:affine-coset-sum}
    \Dcount{\Lambda}{K}{r\Gamma} =\sum_{\xi\in\Lambda/\Gamma}\Ncount{r}{A_\xi;\Gamma} \qquad(r\in\Z_{>0}).
\end{equation}
\end{lemma}

\begin{proof}
Congruence modulo $r\Gamma$ never identifies points in different $\Gamma$-cosets, which gives \eqref{eq:affine-coset-sum}. Moreover, $\conv(A_\xi)+z_\xi\subseteq K$, so $\Gamma\cap\conv(A_\xi)=A_\xi$. Changing $z_\xi$ translates $A_\xi$ by an element of $\Gamma$ and preserves its residue counts. Finiteness follows from the compactness of $K$ and the discreteness of $\Lambda$.
\end{proof}

It suffices to compare consecutive moduli. For a nonempty set in rank $i$, $\Ncount{1}{A}=1$ and $\Ncount{2}{A}\leq2^i$, so $1\to2$ is immediate.

\begin{lemma}[Line formula]\label{lem:line-formula}
Let $L$ have rank $i\geq1$, and let $A=\conv(A)\cap L$ be a nonempty finite set of affine dimension at most one. If $n=|A|$, then
$$\Ncount{q}{A;L}=\min(n,q) \qquad(q\in\Z_{>0}).$$
Consequently, $\Ncount{q}{A;L}/q^i$ is nonincreasing in $q$.
\end{lemma}

\begin{proof}
The formula is immediate for a singleton. Otherwise, translate one point of $A$ to the origin and put $H=\lin_{\R}(A-A)$ and $L_H=L\cap H$. The set $A$ consists of $n$ consecutive points of the rank-one lattice $L_H$. Since
$$qL\cap H=q(L\cap H)=qL_H,$$
congruence on $A$ modulo $qL$ is the same as congruence modulo $qL_H$. This proves the formula. Finally,
$$\frac{\min(n,q)}{q^i}=q^{1-i}\min\{n/q,1\}$$
is a product of two nonnegative nonincreasing functions.
\end{proof}

Together with Lemma~\ref{lem:affine-cosets}, this proves the rank-one case of Theorem~\ref{thm:main} and handles planar sets of affine dimension at most one. For rank zero, the residue count is independent of the modulus.

We next prove the comparison from modulus $2$ to modulus $3$. The argument turns a missing parity class into three missing classes modulo $3L$.

\begin{proposition}[Full-lattice base step]\label{prop:base-step}
Let $A=\conv(A)\cap L$ be a finite lattice-convex set in a rank-two lattice. Then
$$\frac{\Ncount{2}{A}}4\geq\frac{\Ncount{3}{A}}9.$$
\end{proposition}

\begin{proof}
The empty case is immediate, and Lemma~\ref{lem:line-formula} handles sets of affine dimension at most one. Suppose that $P=\conv(A)$ is two-dimensional. A triangle of least positive area with vertices in $A$ has no other lattice points. By Pick's theorem, its two edge vectors at a vertex form a basis of $L$, so $\Ncount{2}{A}\geq3$. The case $\Ncount{2}{A}=4$ is immediate. It remains to show that
\begin{equation}\label{eq:base-three-six}
    \Ncount{2}{A}=3\quad\Longrightarrow\quad \Ncount{3}{A}\leq6.
\end{equation}

Choose a missing class $u+2L$ and put $Q=(P-u)/2$. Since $Q$ is compact and disjoint from $L$, slightly enlarge it without meeting $L$ and extend the enlarged set to a maximal lattice-free convex set $M$. Thus $Q\subset\operatorname{int}M$ and $\operatorname{int}M\cap L=\varnothing$. By \cite[Theorem~2.2 and Corollary~2.2]{BasuConfortiCornuejolsZambelli2010}, $M$ is a strip with lattice boundary lines or a bounded polygon with a lattice point in the relative interior of every edge.

Choose $z_1,z_2,z_3\in\partial M\cap L$ distinct modulo $3L$ as follows. For a strip, take three consecutive lattice points on one boundary line. For a polygon, take lattice points in the relative interiors of three different edges. In the polygon case, a congruence $z_i\equiv z_j\pmod{3L}$ with $i\ne j$ would give
$$\frac{2z_i+z_j}{3}\in\operatorname{int}M\cap L,$$
a contradiction.

If $p\in P\cap(u+2z_i+3L)$, then
$$\frac{p-u+z_i}{3}=\frac23\frac{p-u}{2}+\frac13z_i\in\operatorname{int}M\cap L,$$
a contradiction. Thus $P$ misses the three distinct classes $u+2z_i+3L$, and $\Ncount{3}{A}\leq6$.
\end{proof}
\section{Residue representatives and lattice boxes}\label{sec:terminal}

We work in a rank-two lattice $L$ and write $\Ncount{q}{A}$ for $\Ncount{q}{A;L}$. For $x\in L$ and primitive $v\in L$, the segment $[x,x+mv]$, $m\in\Z_{\geq0}$, has \emph{lattice length} $m$.

Our aim is to compare $r$ and $s=r+1$. We first replace $A$ by a lattice-convex subset containing exactly one point in each class represented modulo $sL$.

\begin{lemma}[Residue-preserving reduction]\label{lem:terminal-peeling}
Let $L$ be a rank-two lattice, let $A=\conv(A)\cap L$ be finite and nonempty, and let $s$ be a positive integer. There is a lattice-convex set $B\subseteq A$ such that
\begin{equation}\label{eq:terminal-peeling}
    \lvert B\rvert=\Ncount{s}{B}=\Ncount{s}{A}.
\end{equation}
\end{lemma}

\begin{proof}
Suppose that two points of $A$ differ by $su\neq0$, and let $v$ be the primitive direction of $u$. The polygon $P=\conv(A)$ has a chord of lattice length at least $s$ in direction $v$. If $P$ is a segment, choose an endpoint $z$ from which $P$ contains $s$ primitive steps in direction $v$. Otherwise, the longest-chord lemma of B\'ar\'any and F\"uredi \cite[Lemma~1]{BaranyFuredi2001} supplies a longest chord in this primitive direction having a lattice vertex $z$ of $P$ as one endpoint. Orient $v$ into the chord. Then $z,z+sv\in A$.

Delete $z$. Since $z$ is a vertex,
$$A\setminus\{z\}=\conv(A\setminus\{z\})\cap L,$$
so lattice convexity is preserved. The class of $z$ modulo $sL$ remains represented by $z+sv$. Replace $P$ by the convex hull of the current set and iterate. Each deletion removes one point, so the process terminates with reduction modulo $sL$ injective, giving \eqref{eq:terminal-peeling}.
\end{proof}

Since $B\subseteq A$, we also have $\Ncount{r}{B}\leq\Ncount{r}{A}$ for every $r$. Thus a comparison for $B$ transfers to $A$:
$$\frac{\Ncount{r}{A}}{r^2} \geq\frac{\Ncount{r}{B}}{r^2} \geq\frac{\Ncount{s}{B}}{s^2} =\frac{\Ncount{s}{A}}{s^2}.$$

For the adjacent comparison, let $s=r+1$ and suppose that $A$ is injective modulo $sL$. Put
\begin{equation}\label{eq:defect-definition}
    t:=\lvert A\rvert-\Ncount{r}{A} =\Ncount{s}{A}-\Ncount{r}{A}.
\end{equation}
Since $s^2-r^2=2r+1$, the adjacent comparison is equivalent to
\begin{equation}\label{eq:defect-criterion}
    (2r+1)\Ncount{r}{A}\geq r^2t.
\end{equation}
We next place $A$ in a lattice box. For $s=r+1$, this gives the four-corner description of the fibers modulo $rL$ in Corollary~\ref{cor:four-block}.

\begin{lemma}[Lattice-box containment]\label{lem:lattice-box}
Let $s$ be a positive integer, let $L$ be a rank-two lattice, and let $P$ be a two-dimensional lattice polygon. Suppose that $P\cap L$ is injective modulo $sL$. There are $z\in L$, a basis $e_1,e_2$ of $L$, and nonnegative integers $m_1,m_2$ such that
\begin{equation}\label{eq:sharp-box}
    P\subseteq z+[0,m_1]e_1+[0,m_2]e_2, \qquad m_i\leq\left\lfloor\frac{3s-1}{2}\right\rfloor.
\end{equation}
\end{lemma}

\begin{proof}
For a linear form $f:\lin_{\R}L\to\R$ with $f(L)\subseteq\Z$, put
$$w(f)=\max_{x\in P}f(x)-\min_{x\in P}f(x).$$
The longest-chord argument in Lemma~\ref{lem:terminal-peeling} gives
\begin{equation}\label{eq:lambda-one}
    \lambda_1(P-P,L)>\frac1s.
\end{equation}
The planar transference inequality of Ambro and Ito \cite[Theorem~6.4]{AmbroIto2020} therefore gives two linearly independent forms $f_1,f_2$ satisfying $w(f_i)<3s/2$.

Fix a basis of $L$ and compute determinants using the coefficient vectors of the forms in this basis. Among the independent pairs satisfying the width bounds, choose $f_1,f_2$ to minimize the positive integer $\Delta=|\det(f_1,f_2)|$. If $\Delta>1$, the centered parallelogram spanned by $f_1,f_2$ contains a nonzero form $g$ taking integer values on $L$, with
$$g=\alpha f_1+\beta f_2,\qquad |\alpha|,|\beta|\leq\frac12.$$
Then
$$w(g)\leq|\alpha|w(f_1)+|\beta|w(f_2)<3s/2.$$
Moreover,
$$|\det(f_1,g)|=|\beta|\Delta,\qquad |\det(g,f_2)|=|\alpha|\Delta.$$
Since $g\ne0$, at least one of the pairs $(f_1,g)$ and $(g,f_2)$ is independent. This pair satisfies the width bounds and has absolute determinant at most $\Delta/2<\Delta$, contradicting the choice of $f_1,f_2$. Thus $\Delta=1$, and the vectors dual to $f_1,f_2$ form a basis $e_1,e_2$ of $L$.

Put $a_i=\min_{x\in P}f_i(x)$ and $m_i=w(f_i)$. These are integers, since the extrema occur at lattice vertices. For $z=a_1e_1+a_2e_2$, we obtain $P\subseteq z+[0,m_1]e_1+[0,m_2]e_2$ and $m_i<3s/2$, which gives \eqref{eq:sharp-box}.
\end{proof}

For $r\geq3$ and $s=r+1$, the bound in \eqref{eq:sharp-box} is at most $2r-1$. We use the basis supplied by Lemma~\ref{lem:lattice-box}. Inside this box, the points of a fixed residue class modulo $rL$ have at most two values in either coordinate, separated by $r$. They therefore lie among the four corners of a translate of $[0,r]e_1+[0,r]e_2$.

\begin{corollary}[Four-corner fibers]\label{cor:four-block}
Let $r\geq3$, let $s=r+1$, and let $P$ be a two-dimensional lattice polygon in a rank-two lattice $L$. Suppose that $A=P\cap L$ is injective modulo $sL$. In the basis of Lemma~\ref{lem:lattice-box}, every fiber $F=A\cap(a+rL)$ has at most four points, and for each such $F$ there is a $b_F\in L$ such that
\begin{equation}\label{eq:four-block}
    F-b_F\subseteq r\{0,e_1,e_2,e_1+e_2\}.
\end{equation}
If $|F|=4$, then
$$A=\bigl(b_F+[0,r]e_1+[0,r]e_2\bigr)\cap L,$$
and hence
$$\frac{\Ncount{r}{A}}{r^2} =\frac{\Ncount{s}{A}}{s^2}=1.$$
\end{corollary}

\begin{proof}
Only the four-point assertion needs proof. The four corners force the parallelogram $b_F+[0,r]e_1+[0,r]e_2$ into $P$. It contains $(r+1)^2=s^2$ lattice points. Injectivity modulo $sL$ gives $|A|\leq s^2$, so these are all the points of $A$, and they represent all $r^2$ classes modulo $rL$.
\end{proof}
\section{Lattice segments and their intersections}\label{sec:lattice-chords}

Fix $r\geq3$ and $s=r+1$. By Lemmas~\ref{lem:terminal-peeling} and~\ref{lem:line-formula}, we may work with a two-dimensional lattice polygon $P$ in a rank-two lattice $L$ such that $A=P\cap L$ is injective modulo $sL$. We use the basis $e_1,e_2$ of Corollary~\ref{cor:four-block} and assume that no four-point fiber modulo $rL$ occurs, since that corollary already settles the comparison in this case. By \eqref{eq:defect-definition},
\begin{equation}\label{eq:fiber-defect}
    t=\sum_{\substack{\xi\in L/rL\\A\cap\xi\neq\varnothing}} (|A\cap\xi|-1).
\end{equation}
From each two-point fiber, select the segment joining its points. From each three-point fiber, select the joining segment parallel to $e_1$ and the one parallel to $e_2$. A singleton contributes no segment. Thus a fiber $F$ contributes $|F|-1$ segments, at most two, giving exactly $t$ segments in total. Each selected segment has the form
\begin{equation}\label{eq:physical-chord}
    Q=[x,x+rv]\subseteq P, \qquad Q\cap L=\{x,x+v,\ldots,x+rv\}\subseteq A.
\end{equation}
Here $v$ is primitive and, up to sign, is one of
\begin{equation}\label{eq:four-directions}
    e_1,\qquad e_2,\qquad e_1+e_2,\qquad e_1-e_2.
\end{equation}
Selected diagonal segments arise only from two-point fibers.

The image of $Q\cap L$ modulo $rL$ is the line
\begin{equation}\label{eq:torus-support}
    \ell_r(Q):=\{x+jv+rL:j=0,\ldots,r-1\}\subseteq L/rL.
\end{equation}
Since a primitive vector extends to a basis of $L$, it has order $r$ in $L/rL$. Thus $Q$ contains $r+1$ lattice points, whereas $\ell_r(Q)$ has $r$ points.

The next two configurations give $\Ncount{r}{A}=r^2$. The desired comparison then follows from $\Ncount{s}{A}=|A|\leq s^2$.

\begin{lemma}[Repeated images of parallel segments]\label{lem:repeated-support}
Suppose that two distinct parallel selected segments $Q,Q'$ have the same image modulo $rL$. Then $\Ncount{r}{A}=r^2$.
\end{lemma}

\begin{proof}
Choose a basis in which the segment direction is horizontal. The segments cannot lie on the same line, since their convex hull would have lattice length at least $r+1=s$. Thus their horizontal levels are distinct. Equality of their images modulo $rL$ means that these levels differ by $mr$ for some integer $m\geq1$.

By convexity, $\conv(Q\cup Q')\subseteq P$. This convex hull has a horizontal section of length $r$, measured in the chosen lattice coordinates, at every intermediate integer level. Since $mr\geq r$, we can choose $r$ consecutive integer levels on or between the two lines. They give all vertical residues modulo $r$, and every horizontal section of length $r$ contains at least $r$ consecutive lattice points. Thus every horizontal residue also occurs. All lattice points just counted belong to $A=P\cap L$; hence all $r^2$ classes are represented.
\end{proof}

\begin{lemma}[Determinant-two directions]\label{lem:determinant-two}
Let $r\geq3$ be an integer, let $L$ be a rank-two lattice, and let $P$ be a lattice polygon such that $A=P\cap L$ is injective modulo $(r+1)L$. If $P$ contains two segments of lattice length $r$ in primitive directions $u,v$ whose determinant in a basis of $L$ has absolute value two, then $\Ncount{r}{A}=r^2$.
\end{lemma}

\begin{proof}
Injectivity modulo $(r+1)L$ forbids a segment of lattice length $r+1$ in $P$. The classes of $u,v$ modulo $2L$ are nonzero and linearly dependent, so $u\equiv v\pmod{2L}$. Consequently, $(u+v)/2$ and $(u-v)/2$ belong to $L$ and have determinant of absolute value one, hence form a basis of $L$. In these coordinates, after translation, the two segments are
$$[(0,0),(r,r)], \qquad [(a,r+c),(a+r,c)], \qquad a,c\in\Z.$$
Interchanging the coordinates swaps $a,c$. Interchanging the segments and applying $(x,y)\mapsto(x-a,r+c-y)$ replaces $(a,c)$ by $(-a,c)$. These operations allow us to assume $a\geq c\geq0$.

By convexity, $P$ contains the triangle with vertices $(0,0),(a,r+c),(a+r,c)$. Its intersection with $y=c$ contains
$$\left[\frac{ac}{r+c},\,a+r\right]\times\{c\}.$$
The absence of a segment of lattice length $r+1$ therefore gives
$$a+r-\left\lceil\frac{ac}{r+c}\right\rceil\leq r, \qquad\text{hence}\qquad ar<r+c\leq r+a.$$
Thus $(r-1)a<r$, so $a\leq1$. If $a=1$, the strict inequality forces $c=1$. Consequently $(a,c)$ is either $(0,0)$ or $(1,1)$.

In the first case, the convex hull of the two segments is $[0,r]^2$. In the second, $(1,1),(r,r)$ lie on the first segment, and joining $(1,1)$ to the endpoints of the second gives $(r,1),(1,r)\in P$. Hence $[1,r]^2\subseteq P$ in both cases. Its $r^2$ lattice points represent every class modulo $rL$, proving $\Ncount{r}{A}=r^2$.
\end{proof}

Two primitive directions are \emph{unimodular} if they form a basis of $L$. Outside the two cases above, the direction classes are pairwise unimodular: among \eqref{eq:four-directions}, only the two diagonals have determinant of absolute value two. Thus at most three direction classes remain.

\begin{lemma}[Intersections and concurrence]\label{lem:physical-crossing}
Two selected segments whose directions form a basis of $L$ meet in a lattice point. Moreover, if three selected segments have pairwise-unimodular directions and their images modulo $rL$ have a common point, then the three segments meet at one lattice point.
\end{lemma}

\begin{proof}
Modulo $sL$, a primitive vector has order $s$, so the reduction of $Q\cap L$ is a complete affine line of $s$ points. The two affine lines in these basis directions meet in exactly one class modulo $sL$. Each of the two segments contains a representative of that class, and injectivity modulo $sL$ forces the two representatives to coincide.

For the second assertion, let $p_{ij}$ be the intersection of the $i$th and $j$th segments. Pairwise unimodularity implies that all $p_{ij}$ lie in the common class modulo $rL$. If they are not all equal, they are pairwise distinct: equality of two puts one point on all three segments. On each segment, write the difference of its two intersection points as $ku$, where $u$ is its primitive direction and $0<|k|\leq r$. Their congruence modulo $rL$ gives $ku\in rL$. Since $u$ extends to a basis of $L$, this forces $r\mid k$, so $|k|=r$ and the intersection points are the two endpoints. All three segments would therefore have their endpoints in one $rL$-fiber, whereas at most two segments were selected from any fiber.
\end{proof}
\section{Covering and counting on the finite torus}\label{sec:finite-torus}

For the geometric argument below, let $r\geq3$ be an integer. Sections~\ref{sec:terminal} and~\ref{sec:lattice-chords} handle four-point fibers, parallel segments with the same image modulo $rL$, and determinant-two pairs: each forces all classes modulo $rL$ to be represented. In the remaining case, the selected segments have distinct images modulo $rL$ and pairwise-unimodular direction classes. The two diagonal directions in \eqref{eq:four-directions} cannot both occur. Replacing $e_2$ by $-e_2$ if necessary, we may assume that the occurring directions are among
$$e_1,\qquad e_2,\qquad e_1+e_2,$$
up to sign. Identify $L$ with $\Z^2$ and $\lin_{\R}L$ with $\R^2$ using the basis $e_1,e_2$. We use these coordinates both for the segments and for their residue classes. Thus $L/rL$ is identified with
$$G_r:=(\Z/r\Z)^2.$$
Write $\mathcal H,\mathcal V,\mathcal D$ for the three families of lines in $G_r$ given respectively by
$$y=a,\qquad x=b,\qquad x-y=c \qquad(a,b,c\in\Z/r\Z).$$
Each family contains $r$ disjoint lines of $r$ points, and every point lies on one line from each class. If $\mathcal S\subseteq\mathcal H\cup\mathcal V\cup\mathcal D$, put
$$U(\mathcal S):=\bigcup_{\ell\in\mathcal S}\ell.$$
The transformations
$$(x,y)\longmapsto(y,x),\qquad (x,y)\longmapsto(x,x-y)$$
generate every permutation of the three unoriented direction classes, up to a sign change of a line parameter. These integer linear maps preserve lattice lengths, congruences, and intersections.

The main geometric step is to complete the images of the selected segments to a cover with $2r+1$ lines. The construction uses the intersection and concurrence properties established in Lemma~\ref{lem:physical-crossing}.

\begin{proposition}[Completion to a cover]\label{prop:completion}
Let $r\geq3$ be an integer, and let $\mathcal Q$ be a family of segments selected as in Section~\ref{sec:lattice-chords}. Suppose that it uses at most three direction classes and that every pair of distinct classes is unimodular. Put
$$\mathcal S:=\{\ell_r(Q):Q\in\mathcal Q\}.$$
Assume that distinct parallel segments have distinct images modulo $rL$. Then there is a family $\mathcal C$ of exactly $2r+1$ distinct lines from a pairwise-unimodular triple extending the direction classes occurring in $\mathcal Q$ such that
$$\mathcal S\subseteq\mathcal C, \qquad U(\mathcal C)=G_r.$$
\end{proposition}

\begin{proof}
If fewer than three direction classes occur, extend them to a pairwise-unimodular triple; this fixes three ambient classes but adds no selected line. Use the common lattice coordinates described above. Let $u,v,w$ be the numbers of selected horizontal, vertical, and diagonal lines, respectively; some of these numbers may be zero. If, say, $u+v\leq r+1$, include all $r$ lines of the third class and add $r+1-u-v$ unused lines from the first two classes. There are $2r-u-v$ such unused lines, so this gives a cover with exactly $2r+1$ lines.

We may therefore suppose that every pairwise sum among $u,v,w$ is at least $r+2$. In particular, all three selected families are nonempty. Let $Y,X,W\subseteq\Z$ be the integer parameters of the lines containing the selected segments: $y\in Y$, $x\in X$, and $x-y\in W$, respectively. These are their parameters in $\R^2$, before reduction modulo $r$. Put
$$a=\max Y-\min Y,\qquad b=\max X-\min X,\qquad c=\max W-\min W.$$
Each span is at most $r$. For example, any selected vertical segment meets the selected horizontal segments at the smallest and largest parameters by Lemma~\ref{lem:physical-crossing}, and both intersection points lie on that segment of lattice length $r$. Equality would give two parallel segments with the same image modulo $rL$, so
\begin{equation}\label{eq:spans-upper}
    a,b,c\leq r-1.
\end{equation}
Since an integer parameter set of size $u$ has span at least $u-1$, the pairwise-sum assumption gives
\begin{equation}\label{eq:spans-lower}
    a+b\geq r,\qquad a+c\geq r,\qquad b+c\geq r.
\end{equation}

Use the three-class symmetry to make $a$ the smallest span. Apply the translation $(x,y)\mapsto(x-\min Y-\min W,y-\min Y)$. Denote the translated parameter sets again by $Y,X,W$ and put $d=\min X$. Then
\begin{equation}\label{eq:first-chamber}
    Y\subseteq[0,a],\qquad X\subseteq[d,d+b],\qquad W\subseteq[0,c].
\end{equation}
Here and below, intervals denote sets of integers; their residues index torus lines. An interval with upper endpoint below its lower endpoint is empty.

If $d>c$, the selected horizontal segment at level $0$ contains its intersections with the lines $x-y=0$ and $x=d+b$, and therefore $d+b\leq r$. Together with $a+b\geq r$, this gives $d\leq a\leq c$, a contradiction. If $d<a-b$, the selected horizontal segment at level $a$ contains its intersections with $x=d$ and $x-y=c$, so $a+c-d\leq r$; the inequality $b+c\geq r$ then yields $d\geq a-b$, again a contradiction. Hence
\begin{equation}\label{eq:chamber-rough}
    a-b\leq d\leq c.
\end{equation}

We next arrange $0\leq d\leq a$. Since $a\leq r-1$, the intervals in \eqref{eq:critical-cover-intervals} will then lie in $[0,r-1]$.

If $d\geq0$, apply $(x,y)\mapsto(x,x-y)$. The new parameter sets are
$$Y'=W,\qquad X'=X,\qquad W'=Y.$$
Their spans are $c,b,a$, respectively, and $\min X'=d$. By \eqref{eq:chamber-rough}, $0\leq d\leq c$.

If $d<0$, apply $(x,y)\mapsto(-y,-x)$ and translate both coordinates by $d+b$. Then
$$Y'\subseteq[0,b],\quad X'\subseteq[d+b-a,d+b],\quad W'\subseteq[0,c].$$
The new horizontal, vertical, and diagonal spans are $b,a,c$, respectively, and $\min X'=d+b-a$. Since $d<0$ and \eqref{eq:chamber-rough} gives $d\geq a-b$, we have $0\leq d+b-a\leq b$. After relabeling in either case we have
\begin{equation}\label{eq:normalized-chamber}
    Y\subseteq[0,a],\quad X\subseteq[d,d+b],\quad W\subseteq[0,c],\quad 0\leq d\leq a.
\end{equation}

The extremal horizontal segments at levels $0$ and $a$ contain their intersections with the extremal vertical and diagonal segments. Hence
\begin{equation}\label{eq:chord-span-nonstrict}
    d+b\leq r,\qquad a+c-d\leq r.
\end{equation}
Both inequalities are strict. If $d+b=r$, the selected segments on $y=0$, $x=r$, and $x-y=0$ have pairwise intersections $(0,0),(r,0),(r,r)$. Their images modulo $rL$ concur, but the segments do not, contradicting Lemma~\ref{lem:physical-crossing}. If $a+c-d=r$, use the selected segments on $y=a$, $x=d$, and $x-y=c$: their pairwise intersections are $(d,a),(a+c,a),(d,d-c)$, and the same contradiction applies. Therefore
\begin{equation}\label{eq:chord-span-strict}
    d+b\leq r-1,\qquad a+c-d\leq r-1.
\end{equation}

Set
\begin{equation}\label{eq:critical-cover-intervals}
    Y^*=[0,a],\qquad X^*=[d,r-1],\qquad W^*=[0,r-1-a+d].
\end{equation}
These intervals lie in $[0,r-1]$, contain $Y,X,W$ by \eqref{eq:chord-span-strict}, and have total cardinality $(a+1)+(r-d)+(r-a+d)=2r+1$. The corresponding lines cover $G_r$: a point on neither a $Y^*$-line nor an $X^*$-line has
$$y\in[a+1,r-1],\qquad x\in[0,d-1],$$
and thus
$$x-y\pmod r\in[1,r+d-a-2]\subseteq W^*.$$
When one of these displayed intervals is empty, the implication is vacuous. The corresponding torus lines form the required $\mathcal C$. Pulling them back through the affine-unimodular normalizations completes the proof.
\end{proof}

It remains to estimate how many points the selected lines cover. The next lemma applies to every subfamily of a cover of the size just constructed.
\begin{lemma}[Counting in a cover]
\label{lem:critical-cover-expansion}
Let $r\geq3$, and let $\mathcal C$ be a cover of $G_r$ by exactly $2r+1$ distinct lines from $\mathcal H\cup\mathcal V\cup\mathcal D$. For every $\mathcal S\subseteq\mathcal C$,
\begin{equation}\label{eq:critical-cover-expansion}
    |U(\mathcal S)|\geq\frac{r^2|\mathcal S|}{2r+1}.
\end{equation}
\end{lemma}

\begin{proof}
Put $m=|\mathcal S|$, and let $u,v,w$ be its three direction counts. If $m\leq3r/2$, inclusion--exclusion gives
$$|U(\mathcal S)| \geq rm-(uv+uw+vw) \geq rm-\frac{m^2}{3} \geq\frac{rm}{2} \geq\frac{r^2m}{2r+1}.$$

Suppose that $m>3r/2$, and put $M=\lfloor(r-1)/2\rfloor$. We restore the lines of $\mathcal C\setminus\mathcal S$ one at a time. For a current family $\mathcal S\subseteq\mathcal T\subsetneq\mathcal C$, let $a,b,c$ count the lines outside $\mathcal T$ in the three full direction classes. Thus
$$a+b+c=3r-|\mathcal T|<\frac{3r}{2}.$$
Adding a horizontal line of $\mathcal C\setminus\mathcal T$ covers at most $\min\{b,c\}$ new points: both the vertical and diagonal line of a new point must lie outside $\mathcal T$, and each meets the horizontal line once. The other directions have analogous bounds.

If all remaining lines have one direction, the missing lines in the other two classes lie outside $\mathcal C$. Their total number is at most $3r-(2r+1)=r-1$, so a remaining line covers at most $M$ new points. Otherwise, if every remaining line covered at least $M+1$ new points, the bounds from two different directions would imply $a,b,c\geq M+1$. This contradicts
$$a+b+c<\frac{3r}{2}\leq3(M+1).$$
Hence a line adding at most $M$ new points can always be restored. The condition $|\mathcal T|>3r/2$ is preserved throughout. Since the final family covers $G_r$,
$$r^2-|U(\mathcal S)| \leq(2r+1-m)M \leq\frac{r^2(2r+1-m)}{2r+1},$$
where $M\leq(r-1)/2<r^2/(2r+1)$. Rearranging proves the claim.
\end{proof}

\begin{proof}[Proof of Theorem~\ref{thm:main}]
By Section~\ref{sec:reductions}, it suffices to compare moduli $r$ and $s=r+1$ for $r\geq3$ and a nonempty finite lattice-convex set $A$ in a rank-two lattice $L$. By Lemma~\ref{lem:terminal-peeling}, we may assume that reduction modulo $sL$ is injective on $A$. If $A$ has affine dimension at most one, use Lemma~\ref{lem:line-formula}. Otherwise put $P=\conv(A)$ and $t=|A|-\Ncount{r}{A}$. Since $\Ncount{s}{A}=|A|$, the desired comparison is equivalent to
\begin{equation}\label{eq:assembly-target}
    \Ncount{r}{A}\geq\frac{r^2t}{2r+1}.
\end{equation}
If $\Ncount{r}{A}=r^2$, the comparison follows directly from $\Ncount{s}{A}=|A|\leq s^2$. We may therefore assume $\Ncount{r}{A}<r^2$.

Use the coordinates of Corollary~\ref{cor:four-block}. No fiber has four points, since that would give $\Ncount{r}{A}=r^2$. Select $t$ segments as in Section~\ref{sec:lattice-chords}. Lemmas~\ref{lem:repeated-support} and~\ref{lem:determinant-two} exclude equal images of parallel segments and determinant-two pairs. Thus the selected direction classes are pairwise unimodular. Images in different classes are also distinct, since they meet in exactly one point and each has $r>1$ points. Their family $\mathcal S$ therefore has $t$ lines.

Proposition~\ref{prop:completion} gives a cover with $2r+1$ lines containing $\mathcal S$. Lemma~\ref{lem:critical-cover-expansion} gives
$$\Ncount{r}{A}\geq |U(\mathcal S)| \geq\frac{r^2t}{2r+1},$$
where the first inequality holds because every lattice point on a selected segment belongs to $A$. This proves \eqref{eq:assembly-target}.

Lemma~\ref{lem:terminal-peeling} transfers the comparison to the original set. Chaining consecutive comparisons and then summing over the affine cosets in \eqref{eq:affine-coset-sum} proves the theorem.
\end{proof}

\begin{remark}\label{rem:equality}
In Lemma~\ref{lem:critical-cover-expansion}, the last comparison in each case is strict when $\varnothing\ne\mathcal S\subsetneq\mathcal C$. Hence equality in \eqref{eq:critical-cover-expansion} holds precisely for $\mathcal S=\varnothing$ or $\mathcal S=\mathcal C$.

Under the assumptions of Theorem~\ref{thm:main}, put $i=\rank\Gamma$ and $c=\Dcount{\Lambda}{K}{\Gamma}$. For integers $1\leq r<s$,
$$\frac{\Dcount{\Lambda}{K}{r\Gamma}}{r^i} =\frac{\Dcount{\Lambda}{K}{s\Gamma}}{s^i} \quad\Longleftrightarrow\quad \Dcount{\Lambda}{K}{s\Gamma}=s^ic.$$
Thus equal values occur only at the maximum $c$; once the normalized counts fall below $c$, they decrease strictly.

If $c=0$ or $i=0$, the assertion is immediate. Otherwise, Lemma~\ref{lem:affine-cosets} expresses the difference of the normalized counts as a sum of nonnegative differences over the nonempty slices $A_\xi$. Equality forces equality on every slice, and monotonicity then gives
$$s^i\Ncount{s-1}{A_\xi;\Gamma} =(s-1)^i\Ncount{s}{A_\xi;\Gamma}.$$
Since $\gcd(s^i,(s-1)^i)=1$ and $1\leq\Ncount{s}{A_\xi;\Gamma}\leq s^i$, we obtain $\Ncount{s}{A_\xi;\Gamma}=s^i$. Summing over the $c$ nonempty slices proves the forward implication. Conversely, monotonicity and the value $c$ at modulus one give $c\geq\Dcount{\Lambda}{K}{r\Gamma}/r^i \geq\Dcount{\Lambda}{K}{s\Gamma}/s^i=c$.
\end{remark}

\end{document}